\documentclass[leqno,12pt,reqno]{amsart}
\usepackage{amsmath,amsfonts,amsthm,amssymb,amscd,float,color,multirow}
\usepackage{slashbox}
\usepackage{amscd}
\usepackage[mathscr]{eucal}
\usepackage{enumitem}
\usepackage{graphics}

\theoremstyle{plain}
\newtheorem{thm}{Theorem}
\newtheorem{lem}[thm]{Lemma}

\newtheorem{prop}[thm]{Proposition}

  \newtheorem{thma}[thm]{Theorem}
    \newtheorem{corb}[thm]{Corollary}

\theoremstyle{definition}

\newtheorem{quest}[thm]{Question}

\theoremstyle{remark}
\newtheorem{rem}[thm]{Remark}

\newcommand{\N}{{\mathbb N}}
\newcommand{\Z}{{\mathbb Z}}
\newcommand{\Q}{{\mathbb Q}}
\newcommand{\R}{{\mathbb R}}
\newcommand{\C}{{\mathbb C}}
\newcommand{\sbar}{\,\;\vert\,\;}
\newcommand{\longsbar}{\,\;\mathstrut\vrule\,\;}
\newcommand{\BM}{\bar{M}}
\DeclareMathOperator{\rank}{rank}
\DeclareMathOperator{\Ker}{Ker}
\DeclareMathOperator{\im}{Im}
\DeclareMathOperator{\tr}{tr}

\begin{document}
\title[]{On the middle dimensional homology classes of equilateral polygon spaces}
\author[]{Yasuhiko KAMIYAMA}
 \email{kamiyama@sci.u-ryukyu.ac.jp}
\address[]{Department of Mathematics, University of the Ryukyus, 
Nishihara-Cho, Okinawa 903-0213, Japan} 
\subjclass[2000]{Primary 58D29; Secondary 55R80} 
\keywords{Polygon space; involution; homology class}

\begin{abstract}
Let $M_n$ be the configuration space of equilateral polygonal linkages with $n$
vertices in the Euclidean plane ${\mathbb R}^2$. 
We consider the case that $n$ is odd and set $n=2m+1$. 

In spite of the long history of research, the homology classes in $H_{m-1}(M_n;{\mathbb Z})$
are mysterious and not well-understood. 
Let $\tau\colon M_n \to M_n$ be the involution induced by complex conjugation. 
In this paper, we determine the representation matrix of
the homomorphism $\tau_\ast\colon H_{m-1}(M_n;{\mathbb Z}) \to H_{m-1}(M_n;{\mathbb Z})$
with respect to a basis of $H_{m-1}(M_n;{\mathbb Z})$.
\end{abstract}

\maketitle

\section{Introduction} 
\label{intro}
We consider the configuration space $M_n$ of equilateral polygonal linkages with $n \, (n \geq 3)$
vertices, each edge having length $1$ in the Euclidean plane $\R^2$ modulo orientation preserving
isometry group. 
We remark that $M_n$ has the following description:
\[
 M_n=\left\{(z_1,\dots,z_{n-1}) \in (S^1)^{n-1} \longsbar\sum_{i=1}^{n-1}z_i=1\right\}.
\]
Here $z_i \in S^1 \subset \C$ denotes the unit vectors in the direction of the sides of a polygon.

More generally, starting in \cite{H,KM,W}, 
the topology of the configuration space of polygons of arbitrary edge lengths has been considered by many authors. 
For example, the homology groups were determined in \cite{FS}.
The study culminated in the proof by \cite{FHS} and \cite{S} of a conjecture by Kevin Walker
which states that one can recover relative lengths of edges from 
the integral cohomology ring of the configuration space.

While they made clever arguments to distinguish the cohomology rings, 
the explicit structure of the rings is not determined completely. 
For example, consider $M_n$ for odd $n=2m+1$. 
In this case, $M_n$ is a connected closed manifold of dimension $n-3$. 
The homology groups $H_\ast(M_n;\Z)$ are torsion free and the Poincar\'e polynomial is given by
\begin{equation}
 \label{eq1}
  PS(M_{n}) =
  \sum_{q=0}^{m-2} \binom{n-1}{q}t^q+2 \binom{n-1}{m-1}t^{m-1}+\sum_{q=m}^{n-3} \binom{n-1}{q+2}t^q.
\end{equation}
Moreover, for the natural inclusion 
\begin{equation}
\label{eq2}
i\colon M_n \to (S^1)^{n-1}, 
\end{equation}
the induced homomorphism 
\[
 i_{\ast}\colon H_q(M_n;\Z) \to H_q((S^1)^{n-1};\Z)
\]
is an isomorphism for $q \leq m-2$ and an epimorphism for $q=m-1$. (See \cite{FS,KT}.) The following number plays a central role in this paper:
\begin{equation}
\label{eq3}
D_n=\binom{n-1}{m-1}.
\end{equation}
We have $\rank H_{m-1}((S^1)^{n-1};\Z)=D_n$ but $\rank H_{m-1}(M_n;\Z)=2D_n$.
This implies that $H_{m-1}(M_n;\Z)$ contains homology classes which do not come from $(S^1)^{n-1}$. (See also \eqref{eq13} in \S2.) 
The classes cause difficulty in determining the ring $H^\ast (M_n;\Z)$. 

To consider the homology classes, recall that $M_n$ comes with a natural involution 
\begin{equation}
\label{eq4}
\tau\colon M_n \to M_n, \quad \tau(z_1,\dots,z_{n-1})=(\bar{z}_1,\dots,\bar{z}_{n-1}).
\end{equation}
induced by complex conjugation.
It is sometimes guessed that there might be a basis 
\[
\langle u_1,\dots,u_{D_n},v_1,\dots,v_{D_n} \rangle
\]
of $H_{m-1}(M_n;\Z)$ such that $\tau_\ast (u_i)=v_i$ for all $1 \leq i \leq D_n$.
In other words, we fix a basis of $H_{m-1}(M_n;\Z)$ and denote by $A$ the representation matrix of the homomorphism
\begin{equation}
\label{eq5}
\tau_\ast \colon H_{m-1}(M_n;\Z) \to H_{m-1}(M_n;\Z)
\end{equation}
with respect to the basis. Then it is guessed that there might exist $S \in GL(n,\Z)$ such that
\begin{equation}
\label{eq6}
S^{-1}AS= \underset{D_n}{\oplus} \begin{pmatrix}0&1\\1 & 0 \end{pmatrix}.
\end{equation}

Now we pose the following:
\begin{quest}
\label{q1}
Is \eqref{eq6} true for odd $n$?
\end{quest}

The purpose of this paper is to give a negative answer to Question \ref{q1}.
This paper is organized as follows. 
In \S\ref{results} we state our main results (Theorem \ref{thA} and Corollary \ref{coB}).
Theorem \ref{thA} is essentially equivalent to Theorem \ref{th3}.
Theorem \ref{th4} is a key result for Theorem \ref{th3} and 
Proposition \ref{pro6} is a key result for Theorem \ref{th4}. 
In \S\ref{prof6} we prove Proposition \ref{pro6} and 
in \S\ref{prof4} we prove Theorem \ref{th4}. 
In \S\ref{prof3} we prove Theorem \ref{th3}
and in \S\ref{profA} we prove our main results.

\section{Statements of the main results}
\label{results}
Hereafter we assume that $n$ is an odd number which satisfies $n \geq 5$. We set $n=2m+1$. 
\begin{thma} 
\label{thA}
We fix a basis of $H_{m-1}(M_n;\Z)$ and denote by $A$ 
the representation matrix of the homomorphism
\eqref{eq5} with respect to the basis.
We set
\[
\alpha_n = \sum_{i=0}^{m-1}2^{m-1-i}\binom{2i}{i} 
\quad \text{and} \quad
\beta_n= D_n-\alpha_n,
\]
where $D_n$ is defined in \eqref{eq3}. Then there exists $S \in GL(n,\Z)$ such that
\begin{equation}
\label{eq7}
S^{-1}AS= \underset{\alpha_n}{\oplus} 
\begin{pmatrix}
0&1\\1 & 0 
\end{pmatrix}
\oplus \underset{\beta_n}{\oplus} 
\begin{pmatrix} 
1 & 0\\0 & -1 
\end{pmatrix}.
\end{equation}
\end{thma}

\begin{corb} 
 \label{coB}
Question \ref{q1} is false for odd $n$ which is more than or equal to $7$.
\end{corb}

\begin{rem} \begin{enumerate}
\item It is reasonable that Question \ref{q1} is true for $n=5$. 
      In fact, $M_5$ is homeomorphic to a connected closed orientable surface of genus $4$. 
      We number the four handles appropriately and denote them by $h_i$ for $1 \leq i \leq 4$. 
      Then $\tau$ acts on $M_5$ by exchanging
      $h_1$ with $h_3$ and $h_2$ with $h_4$. (See, for example, \cite{KM}.)
\item It is known (see, for example, \cite[A082590]{OEIS}) that $\alpha_n$ 
      satisfies the following equation: 
      \[
       \frac{1}{(1-2x)\sqrt{1-4x}}= \sum_{i=0}^\infty \alpha_{2i+3}x^i.
      \]
\end{enumerate}
\end{rem}

Recall that an integral square matrix $P$ with $P^2=I$ has a normal form. 
(See Lemma \ref{lem8} in \S\ref{profA}.)
Thanks to this, we can restate Theorem \ref{thA} into the following:
\begin{thm} 
\label{th3}
An elementary divisor of the homomorphism 
\begin{equation}
\label{eq8}
1-\tau_\ast \colon H_{m-1}(M_n;\Z) \to H_{m-1}(M_n;\Z)
\end{equation}
is $0, 1$ or $2$. 
Moreover, the numbers of the elementary divisors are given by the following table.
\begin{table}[h]
\caption{The numbers of the elementary divisors of $1-\tau_\ast$}
\begin{center}
\begin{tabular}{|c|c|c|c|}\hline
elementary divisor & $0$ & $1$ & $2$\\ \hline number & $D_n$ & $\alpha_n$ & $\beta_n$\\ \hline
\end{tabular}
\end{center}
\end{table}
\end{thm}

The proof of Theorem \ref{th3} consists of three steps. 
We set 
\[
 \BM_n=M_n/\tau.
\]
Moreover, let $\tau$ act on $S^1$ by antipodal and we set 
\[
 E_n= M_n \times_\tau S^1.
\]
The homology groups $H_\ast(\BM_n;\Z)$ are determined in \cite{K} and 
the cohomology ring $H^\ast (\BM_n;\Z_2)$ is determined in \cite{HK} 
(see also Theorem \ref{th5} below). 
\begin{enumerate}
 \item Using the Gysin sequence of the bundle
       \begin{equation}
	\label{eq9}
	 S^1 \to E_n \to \BM_n,
       \end{equation}
       we determine $H_\ast (E_n;\Z_2)$. 
       Here we recall that when $n$ is odd, $\tau$ is a fixed-point-free involution. 
       Hence \eqref{eq9} is an $S^1$-bundle. 
 \item With some more computations, we determine $H_\ast (E_n;\Z)$.
 \item Using the Wang sequence of the bundle
 \begin{equation}
  \label{eq10}
   M_n \to E_n \to S^1,
 \end{equation}
       we determine elementary divisors of the homomorphism \eqref{eq8}.
\end{enumerate}

The following theorem is a key to proving Theorem \ref{th3}.
\begin{thm}
\label{th4}
\begin{enumerate}
 \item Let $x \in H_\ast (E_n;\Z)$ be a torsion element. 
       Then we have $2x=0$. Therefore, for each $q$, the module $H_q(E_n;\Z)$
       has a form 
       \[
	\underset{a_q}{\oplus} \;\Z \oplus \underset{b_q}{\oplus}\; \Z_2
       \]
       for some $a_q, b_q \in \N \cup \{0\}$.
       \label{th4i}
 \item We set
       \[
	\phi_n=\sum_{\substack{q=0\\ \text{$q$ even}}}^{m-2} \binom{n-1}{q} t^q+ D_n t^{m-1}+
       \sum_{\substack{q=m\\ \text{$q$ odd}}}^{n-4} \binom{n-1}{q+2}t^q.
       \]
       Then we have
       \[
	PS_\Q(E_n)=(1+t) \phi_n.
       \]
       Note that this determines $a_q$ in \ref{th4i}.
\item We set
      \[
       \Gamma(E_n)= \sum_{q=0}^\infty b_q t^q,
      \]
      where $b_q$ is defined in \ref{th4i}. 
      Then we have
      \[
      \Gamma(E_n)=\sum_{\substack{q=0\\ \text{$q$ odd}}}^{m-2} \binom{n-1}{q}t^q 
      + \beta_n t^{m-1} 
      + \sum_{\substack{q=m\\ \text{$q$ even}}}^{n-3} \binom{n-1}{q+2}t^q,
      \]
      where $\beta_n$ is defined in Theorem \ref{thA}.
\end{enumerate}
\end{thm}
We prove Theorem \ref{th4} in \S\ref{prof4}. 
We can prove (i) and (ii) of the theorem by standard arguments. 
But for (iii), we need to prove Proposition \ref{pro6} below.
To state the proposition, we recall the following:
\begin{thm}[{\cite[Cor. 9.2 and Prop. 9.3]{HK}}] 
 \label{th5}
 \begin{enumerate}
  \item The mod $2$ cohomology ring of $\BM_n$ is 
	\[
	H^\ast (\BM_n;\Z_2)=\Z_2[R_n,V_1,\dots,V_{n-1}]/\mathscr{I}_n,
	\]
	where $R_n$ and $V_1,\dots,V_{n-1}$ are of degree $1$ and 
	$\mathscr{I}_n$ is the ideal generated by the three families

	\begin{enumerate}[label={\upshape (R\arabic*)}]
	 \setlength{\itemsep}{4pt}
	 \item $\displaystyle V_i^2+R_nV_i\quad$ where $i=1,\dots,n-1$,\label{R1} 
	 \item $\displaystyle \prod_{i \in S} V_i\quad$ 
	       where $S \subset \{1,\dots,n-1\}$ is such that $m \leq |S|$,\label{R2}
	 \item $\displaystyle \sum_{S \subset L}R_n^{|L-S|-1} \prod_{i \in S} V_i\quad $ 
	       where $L \subset \{1,\dots,n-1\}$ is such that $m+1 \leq |L|$.\label{R3}
	\end{enumerate}

	The symbol $S$ in \ref{R3} runs over all subsets of $L$ including the empty set. 
	By \ref{R2} a term of the sum in \ref{R3} vanishes if $m \leq |S|$.
  \item Let $\xi \to \BM_n$ be the line bundle associated with the regular $2$-cover 
	$\pi\colon M_n \to \BM_n$. 
	Then we have $R_n=w_1(\xi)$.
 \end{enumerate}
\end{thm}

For a field $K$ and a space $X$, let $PS_K(X)$ be the
Poincar\'e polynomial of $X$ with coefficients in $K$. 
As a corollary of Theorem \ref{th5} (i), we have
\begin{equation}
 \label{eq11}
  PS_{\Z_2}(\BM_n) = 
  \sum_{q=0}^{m-1}
  \left( \sum_{i=0}^q \binom{n-1}{i} \right)t^q 
  + \sum_{q=m}^{n-3}\left(\sum_{i=0}^{n-3-q}\binom{n-1}{i} \right) t^q.
\end{equation}

\begin{prop} 
\label{pro6}
We define
\[
 \lambda_q= 
 \rank\left[\cup R_n^2\colon H^q (\BM_n;\Z_2) \to H^{q+2}(\BM_n;\Z_2)\right]. 
\]
We set
\[
 \gamma _n=\sum_{i=0}^{m-2}\binom{n-1}{i}.
\]
Then the following result holds.
\[
 \lambda_q=
 \begin{cases}
  \dim H^q(\BM_n;\Z_2), &\quad 0 \leq q \leq m-3,\\
  \gamma_n-\beta_n, &\quad q=m-2,\\
  \dim H^{q+2}(\BM_n;\Z_2), &\quad m-1 \leq q \leq 2m-4.
 \end{cases}
\]
\end{prop}
The proposition implies that the homomorphism
\begin{equation}
\label{eq12}
\cup R_n^2\colon H^q (\BM_n;\Z_2) \to H^{q+2}(\BM_n;\Z_2)
\end{equation}
is injective for
$0 \leq q \leq m-3$ and surjective for $m-1 \leq q \leq 2m-4$. 
The only difficult case is that $\cup R_n^2$ jumps over $H^{m-1} (\BM_n;\Z_2)$. 
Note that $m-1= \frac{1}{2}\dim\BM_n$.

Additionally, we will reduce Proposition \ref{pro6} to Lemma \ref{lem7} in \S\ref{prof6}, 
which is a problem in linear algebra and combinatorics.

Before we leave this section, we give the following:
\begin{rem}
Note that \eqref{eq2} realizes $M_n$ as a codimension-two submanifold of $(S^1)^{n-1}$.
On the other hand, it is also possible to realize $M_n$ as a codimension-one submanifold of $(S^1)^{n-2}$. 
More precisely, there exists a compact $(n-2)$-dimensional submanifold $W_n$ of
$(S^1)^{n-2}$ with boundary such that $\partial W_n=M_n$. (See \cite{MT} for more details.)
Let $X^{(m-1)}$ be the $(m-1)$-skeleton of $X:=(S^1)^{n-1}$, where 
we give a cell structure of $X$ by the product complex of $S^1 =e^0 \cup e^1$.
Then it is proved in \cite{MT} that there is a homotopy equivalence $W_n \simeq X^{(m-1)}$. 
Consider the homology long exact sequence of the pair $(W_n,M_n)$: 
\begin{equation}
\label{eq13}
\begin{CD}
\to H_m(W_n;\Z) @>>> H_m(W_n,M_n;\Z) @>>> H_{m-1}(M_n;\Z) @>>> H_{m-1}(W_n;\Z) \to 0.\\
 \rotatebox{270}{$\cong$} @. \rotatebox{270}{$\cong$} @. @.\rotatebox{270}{$\cong$} \\
H_m(X^{(m-1)};\Z) @. H^{m-1}(W_n;\Z)@. @.\underset{D_n}{\oplus} \;\Z\\
\rotatebox{270}{$=$} @. \rotatebox{270}{$\cong$} @. @. @.\\
 $0$ @. \underset{D_n}{\oplus} \;\Z\ @. @. @. @.
\end{CD}
\end{equation}
Here we used Lefschetz duality.

The involution \eqref{eq4} naturally extends to an involution 
\begin{equation}
\label{eq14}
\tau: W_n \to W_n.
\end{equation}
In contrast to \eqref{eq4}, the involution \eqref{eq14} has fixed points. In fact, the fixed point set coincides
with the set of critical points of the restriction $f|W_n-M_n$, which is a Morse function, of a certain
function $f: W_n \to \R$. (See \cite{MT} for more details.)

Now we remark that \eqref{eq7} gives information on the direct summands of \eqref{eq13}. 
We also have a question: How does \eqref{eq7} reflect the information on the fixed point set of \eqref{eq14}? 
The author is grateful to the referee for suggesting these interesting comments.
\end{rem}

\section{Proof of Proposition \ref{pro6}}
\label{prof6}
(i) The case for $0 \leq q \leq m-3$. 
Among the relations in Theorem \ref{th5},  
only \ref{R1} is effective for $H^{q+2}(\BM_n;\Z_2)$. 
Hence \eqref{eq12} is injective.

(ii) The case for $m-1 \leq q \leq 2m-4$. 
\ref{R2} tells us that any element $x \in H^{q+2}(\BM_n;\Z_2)$ is of the form 
$x=R_n^{q-m+3}y$ for some $y \in \Z_2[R_n,V_1,\dots,V_{n-1}]$. 
Since $q-m+3 \geq 2$, \eqref{eq12} is surjective.

For the rest of this section, we consider the case $q=m-2$. 
For that purpose, we need to study the ideal $\mathscr{I}_n$ more carefully.
We number $(m+1)$-element subsets of $\{1,\cdots,n-1\}$ as $L_1,\cdots,L_{D_n}$, 
where we understand $D_n$ in \eqref{eq3} as $\binom{n-1}{m+1}$. 
For each $L_i$, we consider Theorem \ref{th5} \ref{R3} and write the relation as 
\begin{equation*}
p_{i}=f_{i}, 
\end{equation*}
where $p_{i}$ is the sum of monomials defined by $(m-1)$-element subsets of $L_i$ and 
$f_{i}$ is the sum of the monomials defined by $k$-element subsets of $L_i$ 
for all $k \leq m-2$. 
We set
\[
X_{m+1}=
\left\{\xi=(\xi_1,\cdots,\xi_{D_n})\in (\Z_2)^{D_n}\longsbar \sum_{i=1}^{D_n} \xi_i p_{i}=0 \;\; \text{in $\Z_2[R_n,V_1,\dots,V_{n-1}]$}\right\}. 
\]
\begin{lem}
\label{lem7}
\begin{enumerate}
 \item We have $\dim X_{m+1}= \beta_n$.
 \item We fix a basis $\Omega_{m+1}$ of $X_{m+1}$. 
       For $\xi \in \Omega_{m+1}$, 
       we set
       \[
	F_{\xi}=\sum_{i=1}^{D_n} \xi_i f_{i}.
       \]
       Then the elements of the set
       $\left\{ F_{\xi} \sbar \xi \in \Omega_{m+1} \right\}$
       are linearly independent in $\Z_2[R_n,V_1,\dots,V_{n-1}]$.
\end{enumerate}
\end{lem}
\begin{proof}
(i) Consider the $\Z_2$-vector space $C_k$ with $k$-element subsets of $\{1,\cdots,n-1\}$ 
 as a basis.
Let $\partial_k\colon C_k \to C_{k-2}$ 
be the linear map defined on the basis by the formula 
 \[
  \partial_k (\Delta):=\Gamma_1+\dots+\Gamma_{\binom{k}{k-2}},
 \]
 where $\Gamma_i$ are the $(k-2)$-element subsets of $\Delta$. 
 We regard $\langle L_1,\dots,L_{D_n}\rangle$ as a basis of $C_{m+1}$. 
 Forgetting $R_n$, we regard $p_{i}$ as an element of $C_{m-1}$. 
 From the definition of $p_{i}$, we have $\partial_{m+1} (L_i)=p_{i}$. 
 Thus there is an isomorphism $X_{m+1}\cong\Ker \partial_{m+1}$.

Consider the chain complex
\begin{equation}
 \label{eq15}
 \begin{CD}
  \cdots @>\partial_{m+3}>> C_{m+1} @>\partial_{m+1}>> 
  C_{m-1} @>\partial_{m-1}>> \cdots @>\partial_{\varepsilon+2}>>
  C_\varepsilon @>>>0,
 \end{CD}
\end{equation}
where $\varepsilon$ is $0$ or $1$ according as $m$ is odd or even. 
 Let $H_q$ be the $q$th homology group of \eqref{eq15}.
 Then \cite[Theorem 3.2]{BJS} tells us that 
 $H_{\varepsilon+2i}=0$ for all $i \in \N \cup\{0\}$. 
 In fact, the {\it middle term} in the sense of \cite[p.557]{BJS} is 
 $C_m$ in the other chain complex
 \[
 \begin{CD}
  \cdots @>\partial_{m+2}>> C_{m} @>\partial_{m}>> 
  C_{m-2} @>\partial_{m-2}>> \cdots @>\partial_{\delta+2}>>
  C_\delta @>>>0
 \end{CD} 
 \]
 and this is the only term which gives $H_q\not=0$.

We construct the following long exact sequence from \eqref{eq15}:
\begin{equation*}
 \begin{CD}
 0 @>>> \Ker \partial_{m+1}@>>> C_{m+1} @>\partial_{m+1}>> 
  C_{m-1} @>\partial_{m-1}>> \cdots @>\partial_{\varepsilon+2}>>
  C_\varepsilon @>>>0.
 \end{CD}
\end{equation*}
Computing the Euler characteristic of this, we have 
 \[
  \dim\Ker \partial_{m+1}=D_n-d_{n},
 \]
 where we set
 \[
d_{n}=
\sum_{i=0}^{\lfloor (m-1)/2 \rfloor} (-1)^i \binom{2m}{m-1-2i}
 \]
 for $n=2m+1$.
 It is easy to see (see also \cite[A082590]{OEIS}) that
 \[
  \sum_{k=0}^{m}\binom{2m}{k} \sin\frac{m-k}{2}\pi =\sum_{k=0}^{m-1}2^{m-1-k}\binom{2k}{k}.
 \]
 We have from this that $d_{n}=\alpha_n$ for all odd numbers $n$. 
 Thus we complete the proof of Lemma \ref{lem7} (i).

(ii) Writing $\xi \in \Omega_{m+1}$ in the column vector, 
 we define a $D_n \times \beta_n$ matrix $P$. We have 
 \begin{equation}
  \label{eq16}
   \begin{pmatrix}
    p_1&\cdots &p_{D_n}
   \end{pmatrix}P= 
   \begin{pmatrix}
    0&\cdots& 0 
   \end{pmatrix}.
 \end{equation}
 Assume that
 \begin{equation}
  \label{eq17}
   \begin{pmatrix}
    f_1&\cdots &f_{D_n}
   \end{pmatrix}P 
   \begin{pmatrix}x_1 \\ \vdots \\ x_{\beta_n}
   \end{pmatrix}=0.
 \end{equation}
 We must show that $x_1= \cdots = x_{\beta_n}=0$.
 Using \eqref{eq16} and \eqref{eq17}, we can write
 \[
   \begin{pmatrix}
   p_1+f_1&\cdots &p_{D_n}+f_{D_n}
   \end{pmatrix}
   P 
 \begin{pmatrix}
 x_1 \\ \vdots \\ x_{\beta_n}
 \end{pmatrix}
 =0.
 \]

 As in (i), we forget $R_n$. 
 Then $p_i$ and $f_i$ are elements of $\underset{k=0}{\overset{m-1}{\oplus}}{C_k}$.
 Using the monomial basis of $\underset{k=0}{\overset{m-1}{\oplus}}{C_k}$, 
 we identify
  $\begin{pmatrix}p_1+f_1&\cdots &p_{D_n}+f_{D_n}\end{pmatrix}$ 
 with an $r \times D_n$ matrix $Q$, 
 where $r=\sum_{k=0}^{m-1}\binom{2m}{k}$
 and an element of $Q$ is in $\Z_2$. 
 Thus we obtain 
 \begin{equation}
  \label{eq18}
   QP 
   \begin{pmatrix}
    x_1 \\ \vdots \\ x_{\beta_n}
   \end{pmatrix}
   =0.
 \end{equation}

 The following result is proved in \cite[Lemma 1]{WI}: 
 For a fixed $v \in \N$, we set $X=\{1,\cdots,v\}$. 
 We also fix $s$ and $t \in \N$.
 For each $s$-element subset $\Delta$ of $X$, 
 we consider the
 sum of $k$-element subsets of $\Delta$, where $k$ runs over $0\leq k \leq t$. 
In the same way as in $Q$, we construct a 
 $\left(\sum_{k=0}^t \binom{v}{k}\right)\times \binom{v}{s}$ matrix $U$, 
 whose elements are in $\Z_2$. 
 Assume that the condition $t \leq s \leq v-t$ is satisfied. 
 Then we have $\rank U=\binom{v}{t}$. 
 As remarked in \cite[line 5]{WI}, 
 the result certainly holds for any field, in particular, for $\Z_2$. 

 Now setting $v=n-1, s=m+1$ and $t=m-1$, 
 we have $\rank Q=\binom{2m}{m-1}=D_n$. 
 Therefore, the columns of $Q$ are linearly independent. 
 From \eqref{eq18}, we have
 \[
   P \begin{pmatrix}x_1 \\ \vdots \\ x_{\beta_n}\end{pmatrix}=0.
 \]
 Since $\rank P=\beta_n$ by (i), 
 we have $x_1= \cdots = x_{\beta_n}=0$. 
 This completes the proof of Lemma \ref{lem7} (ii).
\end{proof}

\begin{proof}[Proof of Proposition \ref{pro6} for $q=m-2$]
 We define a linear subspace $Y_m$ of $\Z_2[R_n,V_1,\dots,V_{n-1}]$ by 
 \[
  Y_m=
 \{R_n^{m-k}V_{i_1} \cdots V_{i_k} \sbar 
 1 \leq i_1 < \cdots < i_k \leq n-1 
 \text{ and } 0 \leq k \leq m-2\}.
 \]
 Let $Z_m$ be the linear subspace of $\Z_2[R_n,V_1,\dots,V_{n-1}]$ 
 spanned by polynomials in  $\left\{ F_{\xi} \sbar \xi \in \Omega_{m+1} \right\}$, 
 where $F_{\xi}$ is defined in Lemma \ref{lem7} (ii). 

Then we have
 \[
  \im \left[\cup R_n^2\colon H^{m-2} (\BM_n;\Z_2) \to H^{m}(\BM_n;\Z_2)\right]
 \cong Y_{m}/Z_m.
 \]
 Note that $\dim Y_m =\gamma_n$. 
 On the other hand, Lemma \ref{lem7} tells us that $\dim Z_m=\beta_n$.
 Hence $\lambda_{m-2}=\gamma_n-\beta_n$.
\end{proof}

\section{Proof of Theorem \ref{th4}}
\label{prof4}
 (i) We claim the following assertion: 
 Let $\pi\colon X \to \bar{X}$ be a regular $2$-cover 
 such that $H_\ast (X;\Z)$ are torsion free. 
 If $x \in H_\ast (\bar{X};\Z)$ is a torsion element, 
 then we have $2x=0$. 

 The assertion is well-known. 
 (See, for example, the arguments in \cite[9.3.2]{DIK}.) 
 In fact, if $\tr^\ast \colon H_\ast(\bar{X};\Z) \to H_\ast(X;\Z)$ is the transfer
homomorphism, then $\pi_\ast \circ \tr^\ast=2$. 

Since $H_\ast (X;\Z)$ are torsion free, 
we have $\tr^\ast (x)=0$. Hence $2x=0$.

Now if we apply the above assertion to the regular $2$-cover 
\[
 \pi\colon M_n \times S^1 \to E_n,
\]
then Theorem \ref{th4} (i) follows. 

(ii) We have
\begin{align}
 H_\ast (E_n;\Q) \cong &\, H_\ast (M_n\times S^1;\Q)^{\tau_\ast} \notag\\
 \cong & \, H_\ast(M_n;\Q)^{\tau_\ast} \otimes H_\ast(S^1;\Q) \notag\\
 \cong & \, H_\ast (\BM_n;\Q) \otimes H_\ast(S^1;\Q). \label{eq19}
\end{align}
It is known in \cite[Theorem C]{K} (see also \cite{HK}) that 
$PS_\Q (\BM_n)=\phi_n$. 
Hence \eqref{eq19} is as given in Theorem \ref{th4} (ii).

(iii) We claim that
\begin{multline}
 \label{eq20}
 PS_{\Z_2}(E_n)=\; \sum_{q=0}^{m-2}\binom{n}{q}t^q 
 +\left(\binom{n}{m-1}+\beta_n\right)\left(t^{m-1}+t^m\right)\\
  +\sum_{q=m+1}^{n-2}\binom{n}{q+2}t^q. 
\end{multline}

To prove this, we consider the Gysin sequence of the bundle \eqref{eq9}:
\[
\begin{CD}
 \cdots @>>> H^{q-2}(\BM_n;\Z_2) @>\cup R_n^{2} >> 
 H^{q}(\BM_n;\Z_2) @>>> H^{q}(E_n;\Z_2)\\
 {} @>>> H^{q-1}(\BM_n;\Z_2) @>\cup R_n^{2}>> \cdots
\end{CD}
\]
To check that the homomorphism is certainly $\cup R_n^2$, 
note that \eqref{eq9} is the sphere bundle associated 
to the vector bundle $M_n\times_\tau \R^2\cong \xi\oplus\xi$. 
Using Theorem \ref{th5} (ii), we have $w_2(\xi \oplus \xi)=w_1(\xi)^2=R_n^2$. 
Hence the homomorphism is $\cup R_n^2$. 

From the exactness, we have
\[
 \dim H^q (E_n;\Z_2)=
 \dim H^{q-1}(\BM_n;\Z_2)+\dim H^q(\BM_n;\Z_2)
 -\lambda_{q-2}-\lambda_{q-1}.
\]
Using \eqref{eq11} and Proposition \ref{pro6}, we obtain \eqref{eq20}.

Now the universal coefficient theorem tells us that
\[
 \Gamma(E_n)=\frac{PS_{\Z_2}(E_n)-PS_\Q(E_n)}{1+t}.
\]
Hence Theorem \ref{th4} (iii) follows.

\section{Proof of Theorem \ref{th3}}
\label{prof3}
Consider the Wang sequence of the bundle \eqref{eq9}
(see, for example, \cite[Lemma 8.4]{M}):
\[
 \begin{CD}
  \cdots @>>> H_{m-1} (M_n;\Z) @>1-\tau_\ast>> 
  H_{m-1} (M_n;\Z) @>>> H_{m-1} (E_n;\Z)\\
  {} @>>> H_{m-2}(M_n;\Z) @>1-\tau_\ast>> 
  H_{m-2} (M_n;\Z) @>>> \cdots
\end{CD}
\]

(i) The case that $m$ is even. 
We claim that the homomorphism 
\begin{equation}
 \label{eq21}
  1-\tau_\ast\colon H_{m-2}(M_n;\Z)   \to H_{m-2}(M_n;\Z)
\end{equation} 
is zero. 
In fact, the following commutative diagram tells us that it will suffice to prove the
assertion for $(S^1)^{n-1}$ and the result is clear.
\[
 \begin{CD}
  H_{m-2}(M_n;\Z) @>1-\tau_\ast>> H_{m-2}(M_n;\Z) \\
  @Vi_\ast V\cong V @V\cong V i_\ast V \\
  H_{m-2}((S^1)^{n-1};\Z) @>>1-\tau_\ast> H_{m-2}((S^1)^{n-1};\Z).
\end{CD}
\]

By \eqref{eq1}, we have
\[
 H_{m-1}(M_n;\Z)=\underset{2D_n}{\oplus}\Z.
\]
Since $m$ is even, Theorem \ref{th4} tells us that
\[
 H_{m-1}(E_n;\Z)=
 \underset{\binom{n-1}{m-2}+D_n}{\oplus} \Z \oplus \underset{\beta_n}{\oplus} \Z_2.
\]
From the exactness, an elementary divisor of the homomorphism \eqref{eq8} is $0, 1$ or $2$. 
The numbers of the elementary divisors $2$ and $0$ are $\beta_n$ and 
\[
 \binom{n-1}{m-2}+D_n-\binom{n-1}{m-2}=D_n,
\]
respectively. 
Hence the number of the elementary divisor $1$ is 
\[
 2D_n-\left(\beta_n+D_n\right)=\alpha_n.
\]

(ii) The case that $m$ is odd. 
The homomorphism \eqref{eq21} satisfies that 
$(1-\tau_\ast)(x)=2x$ for all $x \in H_{m-2}(M_n;\Z)$. 
Using this, we can determine the elementary divisors of \eqref{eq8} in the same way as in (i). 

\section{Proofs of Theorem \ref{thA} and Corollary \ref{coB}}
\label{profA}
\begin{lem}
\label{lem8}
For $x, y$ and $z \in \N \cup \{0\}$, 
we define a square matrix of size $2x+y+z$ by 
\[
 F(x,y,z) 
 = \underset{x}{\oplus} 
 \begin{pmatrix}0&1\\1 & 0 \end{pmatrix}
\oplus \underset{y}{\oplus} (1) {\oplus} \underset{z}{\oplus}(-1).
\]
Then the following results hold. 
\begin{enumerate}
 \item Let $P$ be an integral square matrix such that $P^2=I$. 
       Then there exist $x,y$ and $z \in \N \cup \{0\}$
       such that $P$ and $F(x,y,z)$ are integrally similar, 
       that is, there exists $S \in GL(n,\Z)$ such that
       \[
	S^{-1}PS=F(x,y,z).
       \]
\item If $(x,y,z) \not=(p,q,r)$, then $F(x,y,z)$ and $F(p,q,r)$ are not integrally similar.
\end{enumerate}
\end{lem}
\begin{proof} 
 (i) Using \cite[p.9, Example]{L} (see, in particular, $B_m$ in \cite[p.10]{L}), 
 $P$ is integrally similar to
 \[
 Q:=\begin{pmatrix} I_r & 
     \begin{pmatrix} I_k & 0\\0 & 0\end{pmatrix}\\ 
     0 & -I_s
    \end{pmatrix}
 \]
for some $k$ with $0 \leq k \leq \min(r,s)$. 
Since $Q$ is integrally similar to 
\[
 \underset{k}{\oplus} 
 \begin{pmatrix}0&1\\1 & 0 \end{pmatrix}
\oplus \underset{r-k}{\oplus} (1) {\oplus} \underset{s-k}{\oplus}(-1),
\]
(i) follows.

(ii) If $F(x,y,z)$ and $F(p,q,r)$ are integrally similar, 
 then so are $I-F(x,y,z)$ and $I-F(p,q,r)$. 
 An elementary divisor of $I-F(x,y,z)$ is $0,1$ or $2$ such that 
 their numbers are $x+y, x$ and $z$ respectively, 
 and a similar result holds for $I-F(p,q,r)$. 
 The assumption $(x,y,z) \not= (p,q,r)$ implies that $(x+y,x,z)\not=(p+q,p,r)$. 
 Hence (ii) follows.
\end{proof}

\begin{proof}[Proof of Theorem \ref{thA}] 
 By Lemma \ref{lem8} (i), 
 $A$ is integrally similar to $F(x,y,z)$ for some $x,y$ and $z \in \N \cup\{0\}$.
 Using Theorem \ref{th3}, 
 we compare the numbers of the elementary divisors $0, 1$ and $2$ of $I-A$ and $I-F(x,y,z)$.
 Consequently, we have $x=\alpha_n$ and $y=z=\beta_n$.
\end{proof}

\begin{proof}[Proof of Corollary \ref{coB}] 
 The right-hand side of \eqref{eq6} is $F(d,0,0)$ and 
 that of \eqref{eq7} is $F(\alpha_n,\beta_n, \beta_n)$.
 Since $\beta_n\not=0$ for $n \geq 7$, 
 Lemma \ref{lem8} (ii) tells us that 
 the right-hand sides of \eqref{eq6} and \eqref{eq7} are not integrally similar.
\end{proof}

\end{document}